\documentclass[a4paper,12pt]{article}

\usepackage{color,makeidx,bm,latexsym,amscd,amsmath,amssymb,enumerate,stmaryrd,amsthm,float,graphicx,psfrag,geometry}

\usepackage{tikz,epic,eso-pic}

\tikzset{v/.style={circle, draw, inner sep=2pt, minimum size=6pt, fill=white}}

\theoremstyle{plain}
\newtheorem{theorem}{Theorem}[section]
\newtheorem{corollary}[theorem]{Corollary}
\newtheorem{lemma}[theorem]{Lemma}

\theoremstyle{definition}
\newtheorem{definition}[theorem]{Definition}
\newtheorem{remark}[theorem]{Remark}
\newtheorem{example}[theorem]{Example}
\newtheorem{proposition}[theorem]{Proposition}

\newcommand{\N}{\mathbb{N}}
\newcommand{\Q}{\mathbb{Q}}
\newcommand{\R}{\mathbb{R}}
\newcommand{\Z}{\mathbb{Z}}

\renewcommand{\L}{\mathcal{L}} 
\newcommand{\tilE}{\widetilde{E}} 
\newcommand{\cE}{E}
\newcommand{\Hom}{\operatorname{Hom}} 
\newcommand{\Supp}{\operatorname{Supp}} 
\newcommand{\FTC}{\operatorname{\mathcal{FTC}}} 
\newcommand{\surj}{\operatorname{surj}} 
\newcommand{\ord}{\operatorname{\mathcal{O}}} 
\newcommand{\AOC}{\operatorname{\mathcal{AOC}}} 
\newcommand{\kasen}{\underline{\chi}} 
\newcommand{\id}{\operatorname{id}} 
\newcommand{\ep}{\varepsilon}

\title{Euler characteristic reciprocity for chromatic, flow and order polynomials}
\author{Takahiro Hasebe\thanks{Department of Mathematics, 
Hokkaido University, North 10, West 8, Kita-ku, 
Sapporo 060-0810, JAPAN 
E-mail: thasebe@math.sci.hokudai.ac.jp}, 
Toshinori Miyatani\thanks{Department of Mathematics, 
Hokkaido University, North 10, West 8, Kita-ku, 
Sapporo 060-0810, JAPAN 
E-mail: s153033@math.sci.hokudai.ac.jp}, 
Masahiko Yoshinaga\thanks{Department of Mathematics, 
Hokkaido University, North 10, West 8, Kita-ku, 
Sapporo 060-0810, JAPAN 
E-mail: yoshinaga@math.sci.hokudai.ac.jp}}
\date{\today}
\begin{document}
\maketitle

\begin{abstract} 
The Euler characteristic of a semialgebraic set can be considered as a generalization of the cardinality of a finite set. 
An advantage of semialgebraic sets is that we can define ``negative sets'' to be the sets with negative Euler characteristics. 
Applying this idea to posets, we introduce the notion of semialgebraic posets. 
Using ``negative posets'', 
we establish Stanley's reciprocity theorems for 
order polynomials at the level of Euler characteristics. 
We also formulate the Euler characteristic reciprocities for 
chromatic and flow polynomials. 

\end{abstract}

\tableofcontents

\section{Introduction}
\label{sec:intro}

Let $P$ be a finite poset. 
The \emph{order polynomial} $\ord^\leq(P, t)\in\Q[t]$ and 
the \emph{strict order polynomial} $\ord^<(P, t)\in\Q[t]$ are polynomials which satisfy 
\begin{equation}
\begin{split}
\ord^\leq(P, n)&=\#\Hom^\leq(P, [n]),\\
\ord^<(P, n)&=\#\Hom^<(P, [n]), 
\end{split}
\end{equation}
where $[n]=\{1, \dots, n\}$ with normal ordering and 
\[
\Hom^{\leq(<)}(P, [n])=\{f:P\longrightarrow[n]\mid
x<y\Longrightarrow f(x)\leq(<)f(y)\}
\]
is the set of increasing (resp. strictly increasing) maps. 

These two polynomials are related to each other by 
the following reciprocity theorem proved by 
Stanley (\cite{sta-chr, sta-ord}, see also \cite{ard,bec-rec,bec-san} for a recent survey). 
\begin{equation}
\label{eq:stanrecip}
\ord^<(P, t)=(-1)^{\#P}\cdot\ord^\leq(P, -t). 
\end{equation}
By putting $t=n$, the formula (\ref{eq:stanrecip}) can be informally 
presented as follows. 
\begin{equation}
\label{eq:vague}
\text{`` }
\#\Hom^<(P, [n])=(-1)^{\#P}\cdot
\#\Hom^\leq(P, [-n]). 
\text{ ''}
\end{equation}
It is a natural problem to extend the above reciprocity 
to homomorphisms between arbitrary (finite) posets $P$ and $Q$. 
We may expect a formula of the following type. 
\begin{equation}
\label{eq:nonmath}
\text{`` }
\#\Hom^<(P, Q)=(-1)^{\#P}\cdot\#\Hom^\leq(P, -Q). 
\text{ ''}
\end{equation}
Of course this is not a mathematically justified formula. 
In fact, 
we do not have the notion of a ``negative poset $-Q$.'' 

In \cite{sch-neg}, Schanuel discussed what ``negative sets'' 
should be. A possible answer is that 
a negative set is nothing but a 
semialgebraic set which has a negative Euler characteristic 
(Table \ref{tab:schanuel}). 
\begin{table}[b]
\centering
\begin{tabular}{c||c}
Finite set&Semialgebraic set\\
\hline
Cardinality&Euler characteristic
\end{tabular}
\caption{Negative sets}
\label{tab:schanuel}
\end{table}
For example, the open simplex 
\[
\overset{\circ}{\sigma}_d
=\{(x_1, \dots, x_d)\in\R^d\mid
0<x_1<\cdots<x_d<1\}
\]
has the Euler characteristic 
$e(\overset{\circ}{\sigma}_d)=(-1)^d$, 
and the closed simplex 
\[
\sigma_d
=\{(x_1, \dots, x_d)\in\R^d\mid
0\leq x_1\leq \cdots\leq x_d\leq 1\}
\]
has $e(\sigma_d)=1$. Thus we have the following ``reciprocity''
\begin{equation}
\label{eq:simplex}
e(\overset{\circ}{\sigma}_d)=(-1)^d\cdot e(\sigma_d). 
\end{equation}
This formula looks alike Stanley's reciprocity (\ref{eq:stanrecip}). 
This analogy would indicate that (\ref{eq:stanrecip}) could be explained 
via the computations of Euler characteristic of certain semialgebraic sets. 

In this paper, by introducing the notion of 
\emph{semialgebraic posets}, we settle Euler characteristic 
reciprocity theorems for poset homomorphisms. 
Semialgebraic posets also provide a rigorous formulation for 
the reciprocity (\ref{eq:nonmath}). 
The similar idea works also for reciprocities 
of chromatic and flow polynomials. 

Briefly, a semialgebraic poset $P$ is a semialgebraic set with 
poset structure such that the ordering is defined semialgebraically 
(see Definition \ref{def:saposet}). 
Finite posets and the open interval $(0, 1)\subset\R$ are 
examples of semialgebraic posets. A semialgebraic poset $P$ 
has the Euler characteristic $e(P)\in\Z$ which is an invariant 
of semialgebraic structure of $P$ (see \S \ref{subsec:sas}). 
In particular, if $P$ is a finite poset, then $e(P)=\#P$, and 
if $P$ is the open interval $(0,1)$, then $e((0, 1))=-1$. 

The philosophy presented in the literature \cite{sch-neg} 
suggests to consider the 
``moduli space'' $\Hom^{\leq(<)}(P, Q)$ of poset homomorphisms 
from a finite poset $P$ to a semialgebraic poset $Q$, and then 
computing the Euler characteristic of the moduli space instead 
of counting the number of maps. 

Considering the space $\Hom^{\leq(<)}(P, Q)$ itself and its 
Euler characteristic are not a new idea for chromatic theory of 
finite graphs. For example, in \cite{eas-hug}, the Euler characteristic 
of the space of coloring is explored, and in \cite{yos-chr}, the functorial 
aspects of colorings are studied. 
The essential reason why the Euler characteristic works well is 
the additivity of the Euler characteristic and its consistency with 
inclusion-exclusion principle. 

The point of the present paper is to introduce the negative 
of a poset $Q$ in the category of semialgebraic posets. 
We define $-Q:=Q\times (0, 1)$. Then we have $e(-Q)=-e(Q)$. 
Furthermore, we have the following result. 

\begin{theorem}[Proposition \ref{prop:homspace} and Theorems \ref{thm:erecip}, \ref{thm:eorderp}] 
\label{thm:intro1}
Let $P$ be a finite poset and $Q$ be 
a semialgebraic poset. 
\begin{enumerate}[\rm(i)]
\item\label{i}
$\Hom^\leq(P, Q)$ and 
$\Hom^<(P, Q)$ possess structures of semialgebraic sets. 
\item\label{ii} The following reciprocity of Euler characteristics holds, 
\begin{equation*}
e(\Hom^<(P, \pm Q))=(-1)^{\#P}\cdot e(\Hom^\leq(P, \mp Q)).
\end{equation*}
\item\label{iii} Let $T$ be a semialgebraic totally ordered set. 
Then 
\begin{align*}
&e(\Hom^{\leq}(P, T))=\ord^{\leq}(P, e(T)), \\
&e(\Hom^<(P, T))=\ord^<(P, e(T)). 
\end{align*}

\end{enumerate}
\end{theorem}
The most important result is the second assertion \eqref{ii} which is a rigorous formulation of the reciprocity (\ref{eq:nonmath}). 
It should be emphasized that \eqref{ii} is a substantially new result since $Q$ need not be a totally ordered set. When we specialize to the totally ordered sets $Q=[n]$ and $T=[n]\times (0,1)$, our \eqref{ii} and \eqref{iii} recover 
Stanley's reciprocity \eqref{eq:stanrecip} for order polynomials  
(see \S \ref{subsec:stanrecip}). 

Similar Euler characteristic reciprocities are obtained also for Stanley's chromatic polynomials reciprocity \cite{sta-acy} and for Breuer and Sanyal's flow polynomials reciprocity \cite{bre-san}. 

The paper is organized as follows. In \S \ref{sec:semialgposet}, we 
introduce semialgebraic posets, semialgebraic abelian groups and 
Euler characteristics. In \S \ref{sec:erecip}, we prove the main result, Theorem \ref{thm:intro1} \eqref{ii}. The proof is based on topological (cut and paste) arguments. We also deduce Stanley's reciprocity \eqref{eq:stanrecip} from the main theorem. 
In \S \ref{sec:chroflow}, we describe other Euler characteristic reciprocities, that is, for 
chromatic polynomials of simple graphs and 
flow polynomials of oriented graphs.

\section{Semialgebraic posets and Euler characteristics}

\label{sec:semialgposet}

\subsection{Semialgebraic sets}

\label{subsec:sas}

A subset $X\subset\R^n$ is said to be a semialgebraic set if 
it is expressed as a Boolean connection (i.e.\ a set expressed by a finite combination of $\cup, \cap$ and complements) of subsets of the form 
\[
\{x\in\R^n\mid p(x)>0\}, 
\]
where $p(x)\in\R[x_1, \dots, x_n]$ is a polynomial. Let $f\colon X\longrightarrow Y$ be a map between semialgebraic sets $X\subset\R^n$ 
and $Y\subset\R^m$. It is called semialgebraic if the graph 
\[
\Gamma(f)=\{(x, f(x))\mid x\in X\}\subset\R^{m+n}
\]
is a semialgebraic set. If $f$ is semialgebraic then the pull-back $f^{-1}(Y)$ and the image $f(X)$ are also semialgebraic sets (see \cite{bpr, bcr} for details). 

Any semialgebraic set $X$ has a finite partition into Nash cells, namely, a partition $X=\bigsqcup_{\alpha=1}^k X_\alpha$ such that 
$X_\alpha$ is Nash diffeomorphic (that is a semialgebraic analytic 
diffeomorphism) to the open cell $(0, 1)^{d_\alpha}$ 
for some $d_\alpha\geq 0$. Then the \emph{Euler characteristic} 
\begin{equation}
e(X):=\sum_{\alpha=1}^k(-1)^{d_\alpha}
\end{equation}
is independent of the partition (\cite{cos-ras}). Moreover, the 
Euler characteristic satisfies 
\[
\begin{split}
e(X\sqcup Y)&=e(X)+e(Y), \\
e(X\times Y)&=e(X)\times e(Y). 
\end{split}
\]

\begin{example}
As mentioned in \S \ref{sec:intro}, the closed simplex $\sigma_d$ and 
the open simplex $\overset{\circ}{\sigma}_d$ have 
$e(\sigma_d)=1$ and $e(\overset{\circ}{\sigma}_d)=(-1)^d$. 
\end{example}

\subsection{Semialgebraic posets}

\label{subsec:saposet}

\begin{definition}
\label{def:saposet}
$(P, \leq)$ is called a \emph{semialgebraic poset} if 
\begin{itemize}
\item[(a)] $(P, \leq)$ is a partially ordered set, and 
\item[(b)] there is an injection $i\colon P\hookrightarrow\R^n$ ($n\geq 0$) such 
that the image $i(P)$ is a semialgebraic set and the image of 
\[
\{(x, y)\in P\times P\mid x\leq y\}, 
\]
by the map $i\times i\colon P\times P\longrightarrow\R^n\times\R^n$, 
is also a semialgebraic subset of $\R^n\times\R^n$. 
\end{itemize}
\end{definition}
Let $P$ and $Q$ be semialgebraic posets. The set of 
homomorphisms (strict homomorphisms) of semialgebraic posets 
is defined by 
\begin{equation}
\Hom^{\leq(<)}(P, Q)=
\left\{
f\colon P\longrightarrow Q
\left|
\begin{array}{l}
\text{$f$ is a 
semialgebraic map s.t. }\\
x<y\Longrightarrow f(x)\leq(<)f(y)
\end{array}
\right.
\right\}. 
\end{equation}

\begin{example}
\label{ex:saposet}
\begin{itemize}
\item[(a)] 
A finite poset $(P, \leq)$ admits the structure of a semialgebraic poset, 
since any finite subset in $\R^n$ is a semialgebraic set. A finite poset 
has the Euler characteristic $e(P)=\#P$. 
\item[(b)] 
The open interval $(0, 1)$ and the closed interval $[0, 1]$ are 
semialgebraic posets with respect to the usual ordering induced from $\R$. 
Their Euler characteristics are 
$e((0, 1))=-1$ and $e([0, 1])=1$, respectively. 
\end{itemize}
\end{example}

Let $P$ and $Q$ be posets. 
There are several ways to define poset structures on the product $P\times Q$. 
However, in this paper, we always consider the product $P\times Q$ 
with the lexicographic ordering: 
\[
(p_1, q_1)\leq (p_2, q_2)\Longleftrightarrow
\left\{
\begin{array}{l}
p_1<p_2, \text{ or, }\\
p_1=p_2 \text{ and }q_1\leq q_2, 
\end{array}
\right.
\]
for $(p_i, q_i)\in P\times Q$. 

\begin{proposition}
\label{prop:lex}
Let $P$ and $Q$ be semialgebraic posets. 
Then the product poset 
$P\times Q$ (with lexicographic ordering) admits the structure of a 
semialgebraic poset. 
\end{proposition}

\begin{proof}
Suppose $P\subset\R^n$ and $Q\subset\R^m$. Then 
\[
\begin{split}
&\{((p_1, q_1), (p_2, q_2))\in(P\times Q)^2\mid 
(p_1, q_1)\leq (p_2, q_2)\}\\
&=
\{(p_1, q_1, p_2, q_2)\in(P\times Q)^2 \mid (p_1<p_2)\text{ or }(p_1=p_2 \text{ and }
q_1\leq q_2)\} \\
&\simeq \left(\{(p_1, p_2)\in P^2\mid p_1<p_2\} \times Q^2\right) \sqcup \left( P \times \{(q_1,q_2) \in Q^2\mid q_1\leq q_2\}\right) 
\end{split}
\]
is also semialgebraic since semialgebraicity is preserved by disjoint union, complement and Cartesian products. 
\end{proof}

\begin{proposition}
\label{prop:proj}
Let $P$ and $Q$ be semialgebraic posets. Then the first projection 
$\pi\colon P\times Q\longrightarrow P$ is a homomorphism of semialgebraic 
posets. 
\end{proposition}

\begin{proof}
This is straightforward from the definition of the lexicographic 
ordering. 
\end{proof}
The next result shows that the ``moduli space'' 
of homomorphisms from a finite poset to a semialgebraic poset 
has the structure of a semialgebraic set. 

\begin{proposition}[Theorem \ref{thm:intro1} \eqref{i}]
\label{prop:homspace}
Let $P$ be a finite poset and $Q$ be a semialgebraic poset. 
Then $\Hom^{\leq}(P, Q)$ and $\Hom^<(P, Q)$ have structures  
of semialgebraic sets. 
\end{proposition}

\begin{proof}
Let us set $P=\{p_1, \dots, p_n\}$ and $\L=\{(i, j)\mid p_i<p_j\}$. Since each element $f\in\Hom^{\leq}(P, Q)$ 
can be identified with the tuple $(f(p_1), \dots, f(p_n))\in Q^n$, we have the expression 
\[
\begin{split}
\Hom^{\leq}(P, Q)
&\simeq
\{(q_1, \dots, q_n)\in Q^n\mid q_i\leq q_j\text{ for }(i, j)\in\L\}\\
&=
\bigcap_{(i, j)\in\L}
\{(q_1, \dots, q_n)\in Q^n\mid q_i\leq q_j\}. 
\end{split}
\]
Clearly, the right-hand side is a semialgebraic set. 

The semialgebraicity of $\Hom^<(P, Q)$ is similarly proved. 
\end{proof}

\subsection{Semialgebraic abelian groups}

\label{subsec:semialgabgr}

An abelian group $(\mathcal{A}, +)$ is called a semialgebraic abelian group 
if there exists an injection $i\colon \mathcal{A}\hookrightarrow \R^n$ ($n\geq 0$) 
such that the image $i(\mathcal{A})$ is a semialgebraic set and 
the maps 
\[
\begin{split}
+ &\colon i(\mathcal{A})\times i(\mathcal{A})\longrightarrow i(\mathcal{A}),\ 
(i(x), i(y))\longmapsto i(x+y)\\
(-1) &\colon i(\mathcal{A})\longrightarrow i(\mathcal{A}),\ 
i(x)\longmapsto i(-x)
\end{split}
\]
are semialgebraic maps. 
Finite abelian groups and the set of all real numbers $\R$ 
are semialgebraic abelian groups. 

It is easy to see that if $\mathcal{A}_1$ and $\mathcal{A}_2$ 
are semialgebraic abelian groups, then so is the 
product 
$\mathcal{A}_1\times \mathcal{A}_2$. 

\section{Euler characteristic reciprocity}

\label{sec:erecip}

\subsection{The main result}
\label{subsec:main}

For a semialgebraic poset $Q$, let us define the negative 
by $-Q:=Q\times (0,1)$. Recall that we consider the 
lexicographic ordering on $-Q$. The main theorem of this paper is the following. 
\begin{theorem}[Theorem \ref{thm:intro1} \eqref{ii}]
\label{thm:erecip}
Let $P$ be a finite poset and $Q$ be a semialgebraic poset. 
Then 
\[
e(\Hom^{<}(P, \pm Q))=(-1)^{\#P}\cdot e(\Hom^{\leq}(P, \mp Q)). 
\]
In other words, 
\begin{equation}
\label{eq:main1}
e(\Hom^{<}(P, Q))=(-1)^{\#P}\cdot e(\Hom^{\leq}(P, Q\times (0,1)))
\end{equation}
and 
\begin{equation}
\label{eq:main2}
e(\Hom^{<}(P, Q\times(0,1)))=(-1)^{\#P}\cdot e(\Hom^{\leq}(P, Q))
\end{equation}
hold. 
\end{theorem}
Note that since $-(-Q)\neq Q$, (\ref{eq:main1}) and (\ref{eq:main2}) 
are not equivalent. 

Before the proof of Theorem \ref{thm:erecip}, we present an example which 
illustrates the main idea of the proof. 

\begin{example}
Let $P=Q=\{1, 2\}$ with $1<2$. Clearly we have 
\[
\Hom^<(P, Q)=\{\id\}. 
\]
Let us describe $\Hom^{\leq}(P, Q\times(0, 1))$. Note that 
$Q\times(0,1)$ is isomorphic to the semialgebraic totally ordered set 
$(1, \frac{3}{2})\sqcup(2, \frac{5}{2})$ by the isomorphism 
\[
\varphi\colon Q\times(0, 1)\longrightarrow
\left(1, \frac{3}{2}\right)\sqcup\left(2, \frac{5}{2}\right), 
(a, t)\longmapsto a+\frac{t}{2}. 
\]
A homomorphism $f\in\Hom^{\leq}(P, Q\times(0, 1))$ is described 
by the two values $f(1)=(a_1, t_1)$ and $f(2)=(a_2, t_2)\in Q\times(0, 1)$. 
The condition imposed on $a_1, a_2, t_1$ and $t_2$ (by the inequality 
$f(1)\leq f(2)$) is 
\[
(a_1<a_2), \text{ or } (a_1=a_2 \text{ and }t_1\leq t_2), 
\]
which is equivalent to $a_1+\frac{t_1}{2}\leq a_2+\frac{t_2}{2}$. 
Therefore, the semialgebraic set 
$\Hom^{\leq}(P, Q\times(0, 1))$ 
can be described as in Figure \ref{fig:example}. 
\begin{figure}[htbp]
\centering
\begin{tikzpicture}[scale=1.2]

\draw[->] (-0.5,0)--(4.5,0);
\draw[->] (0,-0.5)--(0,4.5);
\draw[very thin] (-0.5,-0.5)--(4.5,4.5);

\filldraw[fill=gray!20!white, draw=black, dashed, very thin] 
(1,3)--(2,3)--(2,4)-- node[above] {$a_1<a_2$} (1,4)--cycle;

\filldraw[fill=white, draw=black] (1,3) circle (2pt) ;
\filldraw[fill=white, draw=black] (1,4) circle (2pt) ;
\filldraw[fill=white, draw=black] (2,3) circle (2pt) ;
\filldraw[fill=white, draw=black] (2,4) circle (2pt) ;

\filldraw[fill=gray!20!white, draw=black, dashed, very thin] 
(1,1)--(2,2)--(1,2)--cycle;
\draw[very thick] (1,1)--node[right] {$a_1=a_2=1, t_1\leq t_2$} (2,2);
\filldraw[fill=white, draw=black] (1,1) circle (2pt) ;
\filldraw[fill=white, draw=black] (1,2) circle (2pt) ;
\filldraw[fill=white, draw=black] (2,2) circle (2pt) ;

\filldraw[fill=gray!20!white, draw=black, dashed, very thin] 
(3,3)--(4,4)--(3,4)--cycle;
\draw[very thick] (3,3)--node[right] {$a_1=a_2=2, t_1\leq t_2$} (4,4);
\filldraw[fill=white, draw=black] (3,3) circle (2pt) ;
\filldraw[fill=white, draw=black] (4,4) circle (2pt) ;
\filldraw[fill=white, draw=black] (3,4) circle (2pt) ;

\end{tikzpicture}
\caption{$f(1)\leq f(2)$.}
\label{fig:example}
\end{figure}
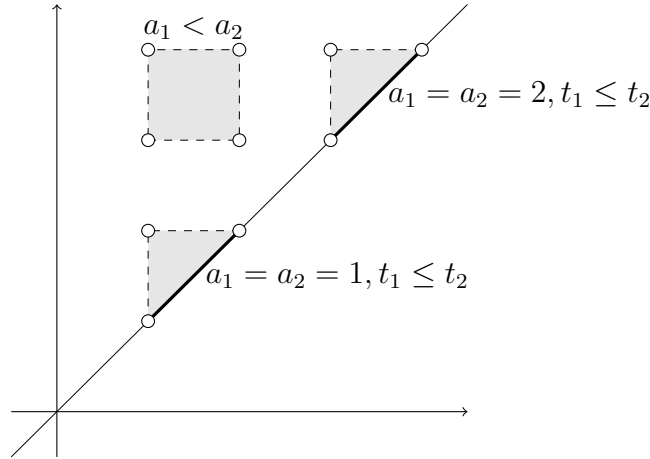
Each diagonal triangle in Figure \ref{fig:example} has a stratification 
$\overset{\circ}{\sigma_2}\sqcup\overset{\circ}{\sigma_1}$. Therefore the 
Euler characteristic is 
$e(\overset{\circ}{\sigma_2}\sqcup\overset{\circ}{\sigma_1})=
e(\overset{\circ}{\sigma_2})+e(\overset{\circ}{\sigma_1})=(-1)^2+(-1)^1=0$. 
On the other hand, the square region corresponding to $a_1<a_2$ has 
the Euler characteristic $(-1)^2=1$. Hence we have 
\[
e(\Hom^{\leq}(P, Q\times (0, 1)))=1=e(\Hom^<(P, Q)). 
\]
\end{example}
The following lemma will be used in the proof of 
Theorem \ref{thm:erecip}. 

\begin{lemma}
\label{lem:semiopen}
Let $P\subset\R^n$ be a $d$-dimensional polytope (i.e., a convex hull of 
a finite set). Fix a hyperplane description 
\[
P=\{\alpha_1\geq 0\}\cap\cdots\cap\{\alpha_N\geq 0\}
\]
of $P$ where $\alpha_i$ are affine maps from $\R^n$ to $\R$. For a given $x_0\in P$, define the associated 
locally closed subset $P_{x_0}$ of $P$ (see Figure \ref{fig:px0}) by 
\[
P_{x_0}=
\bigcap_{\alpha_i(x_0)=0}\{\alpha_i\geq 0\}
\cap
\bigcap_{\alpha_i(x_0)>0}\{\alpha_i> 0\}. 
\]
Then the Euler characteristic is 
\[
e(P_{x_0})=
\left\{
\begin{array}{ll}
(-1)^d, & \text{ if }x_0\in\overset{\circ}{P}\\
0, &\text{ otherwise }(x_0\in\partial P), 
\end{array}
\right.
\]
where $\overset{\circ}{P}$ is the relative interior of $P$ and $\partial P=
P\smallsetminus\overset{\circ}{P}$. 
\begin{figure}[htbp]
\centering
\begin{tikzpicture}[scale=1.2]


\filldraw[fill=gray!20!white, draw=black, dashed, very thin] 
(0,0)--(2,0)--(1,2)-- (0,0);

\filldraw[fill=black, draw=black] (0.5,1) circle (2pt) node[left] {$x_0$};
\draw[very thick] (0,0)--(1,2); 

\filldraw[fill=white, draw=black] (0,0) circle (2pt) ;
\filldraw[fill=white, draw=black] (2,0) circle (2pt) ;
\filldraw[fill=white, draw=black] (1,2) circle (2pt) ;


\filldraw[fill=gray!20!white, draw=black, dashed, very thin] 
(4,0)--(6,0)--(6,1.5)-- (5,2)--(4,1.5)--cycle;

\filldraw[fill=black, draw=black] (5,2) circle (2pt) node[above] {$x_0$};
\draw[very thick] (5,2)--(6,1.5); 
\draw[very thick] (5,2)--(4,1.5); 
\filldraw[fill=white, draw=black] (4,0) circle (2pt) ;
\filldraw[fill=white, draw=black] (6,0) circle (2pt) ;
\filldraw[fill=white, draw=black] (6,1.5) circle (2pt) ;
\filldraw[fill=white, draw=black] (4,1.5) circle (2pt) ;

\end{tikzpicture}
\caption{$P_{x_0}$.}
\label{fig:px0}
\end{figure}
\end{lemma}
\begin{proof}
If $x_0\in\overset{\circ}{P}$, then $P_{x_0}=\overset{\circ}{P}$. 
The Euler characteristic is $e(\overset{\circ}{P})=(-1)^d$. 

Suppose $x_0\in\partial P$. Then $P_{x_0}$ can be expressed as 
\begin{equation}
P_{x_0}=\bigsqcup_{F\ni x_0}\overset{\circ}{F}, 
\end{equation}
where $F$ runs over the faces of $P$ containing $x_0$ and $\overset{\circ}{F}$ 
denotes its relative interior. Then we obtain the decomposition 
\[
P_{x_0}=\overset{\circ}{P}\sqcup\bigsqcup_{F\ni x_0, F\subset\partial P}
\overset{\circ}{F}. 
\]
We look at the structure of the second component 
$Z:=\bigsqcup_{F\ni x_0, F\subset\partial P}\overset{\circ}{F}$. 
For any point 
$y\in Z$, the segment $[x_0, y]$ is contained in $Z$. 
Hence $Z$ is contractible open subset of $\partial P$, which is 
homeomorphic to the $(d-1)$-dimensional open disk. 
The Euler characteristic is computed as 
\[
\begin{split}
e(P_{x_0})&=e(\overset{\circ}{P})+e(Z)\\
&=(-1)^d+(-1)^{d-1}\\
&=0. 
\end{split}
\]
\end{proof}

\subsection{Proof of the main result}
\label{subsec:proof} 

Now we prove Theorem \ref{thm:erecip}. 
Let $\varphi\in\Hom^<(P, Q\times (0, 1))$. Then $\varphi$ is a 
pair of maps 
\[
\varphi=(f, g), 
\]
where $f\colon P\longrightarrow Q$ and $g\colon P\longrightarrow(0, 1)$. 
Let $\pi_1\colon Q\times (0, 1)\longrightarrow Q$ be the first 
projection. Since $\pi_1$ is order-preserving 
(Proposition \ref{prop:proj}), so is $f=\pi_1\circ\varphi$, and hence $f\in\Hom^{\leq}(P, Q)$. 

In order to compute the Euler characteristics, 
we consider the map 
\begin{equation}
\label{eq:importantmap}
\pi_{1*}\colon \Hom^{\leq}(P, Q\times (0,1))
\longrightarrow
\Hom^\leq(P, Q), ~
\varphi\longmapsto\pi_1\circ\varphi=f. 
\end{equation}
Let us set 
\begin{equation}
\begin{split}
M:=&
\Hom^{\leq}(P, Q)\smallsetminus\Hom^<(P, Q)\\
=&
\{f\in\Hom^{\leq}(P, Q)\mid \exists x<y \in P\text{ s.t. }f(x)=f(y)\}. 
\end{split}
\end{equation}
Then obviously, we have 
\begin{equation}
\label{eq:basedecomp}
\Hom^{\leq}(P, Q)=\Hom^<(P, Q)\sqcup M. 
\end{equation}
This decomposition induces that of $\Hom^{\leq}(P, Q\times(0, 1))$, 
\begin{equation}
\label{eq:totaldecomp}
\Hom^{\leq}(P, Q\times(0, 1))=
\pi_{1*}^{-1}\left(\Hom^<(P, Q)\right)\sqcup \pi_{1*}^{-1}(M). 
\end{equation}
By the additivity of the Euler characteristics, we obtain 
\begin{equation}
e\left(\Hom^{\leq}(P, Q\times(0, 1))\right)=
e\left(\pi_{1*}^{-1}\left(\Hom^<(P, Q)\right)\right)+e(\pi_{1*}^{-1}(M)). 
\end{equation}
We claim the following two equalities which are sufficient for the proof 
of (\ref{eq:main1}). 
\begin{eqnarray}
e\left(\pi_{1*}^{-1}\left(\Hom^<(P, Q)\right)\right)&=&
(-1)^{\#P}\cdot e\left(\Hom^<(P, Q)\right), 
\label{eqn:1}\\
e(\pi_{1*}^{-1}(M))&=&0. \label{eqn:2}
\end{eqnarray}
We first prove (\ref{eqn:1}). 
Let $\varphi\in \pi_{1*}^{-1}\left(\Hom^<(P, Q)\right)$, 
that is $\varphi=(f, g)$ with $f\in \Hom^<(P, Q)$. 
By the definition of the ordering of $Q\times (0, 1)$, $(f, g)$ is 
contained in $\pi_{1*}^{-1}\left(\Hom^<(P, Q)\right)$ 
for arbitrary map $g\colon P\longrightarrow (0, 1)$. This implies 
\begin{equation}
\pi_{1*}^{-1}\left(\Hom^<(P, Q)\right)\simeq
\Hom^<(P, Q)\times (0, 1)^{\#P}, 
\end{equation}
which yields (\ref{eqn:1}). 

The proof of (\ref{eqn:2}) requires further stratification of $M$. 
Let 
\[
\L(P):=\{(p_1, p_2)\in P\times P\mid p_1<p_2\}. 
\]
For given $f\in M$, consider the set of collapsing pairs, 
\[
K(f):=
\{(p_1, p_2)\in \L(P)\mid f(p_1)=f(p_2)\}. 
\]
Note that $f\in M$ if and only if $K(f)\neq\emptyset$. 
We decompose $M$ according to $K(f)$. Namely, for any nonempty 
subset $X\subset\L(P)$, define a subset $M_X\subset M$ by 
\[
M_X:=\{f\in M\mid K(f)=X\}. 
\]
Since $\L(P)$ is a finite set, 
\begin{equation}
M=\bigsqcup_{\substack{X\subset\L(P)\\ X \neq \emptyset}}M_X
\end{equation}
is a decomposition of $M$ into finitely many 
semialgebraic sets. Therefore, we obtain 
\[
e(\pi_{1*}^{-1}(M))=\sum_{\substack{X\subset\L(P)\\X\neq\emptyset}}
e(\pi_{1*}^{-1}(M_X)). 
\]
Thus it is enough to show 
$e(\pi_{1*}^{-1}(M_X))=0$ for all $X\subset \L(P)$ as long as 
$\pi_{1*}^{-1}(M_X)\neq\emptyset$ (note that $\pi_{1*}^{-1}(M_X)=\emptyset$ can occur for a nonempty $X$ e.g.\ when $\#Q=1$). 

Now we fix $X\subset\L(P)$ such that $\pi_{1*}^{-1}(M_X)\neq\emptyset$. 
Then we can show that $\pi_{1*}^{-1}(M_X)\longrightarrow M_X$ is a 
trivial fibration. Indeed, for any $f\in M_X$, the condition imposed on $g$ 
by $(f, g)\in\Hom^{\leq}(P, Q\times (0, 1))$ is 
\[
(p_1, p_2)\in X\Longrightarrow g(p_1)\leq g(p_2). 
\]
Hence the fiber $\pi_{1*}^{-1}(f)$ is independent of $f\in M_X$ and 
isomorphic to 
\begin{equation}
\label{eq:fiber}
F_X:=\{(t_p)_{p\in P}\in(0, 1)^P\mid 
(p_1, p_2)\in X\Longrightarrow t_{p_1}\leq t_{p_2}\}, 
\end{equation}
and we have 
\begin{equation}
\label{eq:mxprod}
\pi_{1*}^{-1}(M_X)\simeq M_X\times F_X. 
\end{equation}
The fiber $F_X$ is a locally closed polytope 
defined by the following inequalities. 
\[
0<t_p<1, t_{p_1}\leq t_{p_2} \text{ for }(p_1, p_2)\in X. 
\]
The closure $\overline{F_X}$ is defined by 
\[
\overline{F_X}=
\{(t_p)_{p\in P} \in[0,1]^P\mid t_{p_1}\leq t_{p_2} 
\text{ for }(p_1, p_2)\in X\}. 
\]
Then $F_X$ is equal to the locally closed polytope $(\overline{F_X})_{x_0}$ 
associated to the point $x_0=(\frac{1}{2}, \frac{1}{2}, \dots, \frac{1}{2})\in\partial\overline{F_X}$. 
Since $X\neq\emptyset$, $x_0$ is not contained in the interior of 
$\overline{F_X}$. By Lemma \ref{lem:semiopen}, $e(F_X)=0$. 
Together with (\ref{eq:mxprod}), we conclude 
$e(\pi_{1*}^{-1}(M_X))=0$. This completes the proof of 
(\ref{eq:main1}) of Theorem \ref{thm:erecip}. 

The proof of the other formula (\ref{eq:main2}) is similar to and actually simpler than that of 
(\ref{eq:main1}) since we do not need Lemma \ref{lem:semiopen}.  Again the first projection 
$\pi_{1}\colon Q\times (0, 1)\longmapsto Q$ 
induces the map 
\[
\pi_{1*}\colon \Hom^<(P, Q\times (0, 1))\longrightarrow 
\Hom^{\leq}(P, Q). 
\]
We can prove that this map is surjective and each fiber of $\pi_{1*}^{-1}(M_X)$ (now $X=\emptyset$ is allowed) is isomorphic to 
\[
\overset{\circ}{F_X}=\{(t_p)_{p\in P} \in (0,1)^{P}\mid t_{p_1} < t_{p_2} \text{~for all~}(p_1,p_2)\in X \}. 
\]
This fiber is an open polytope of dimension $\#P$ and hence is isomorphic to $(0,1)^{\#P}$ whose Euler characteristic is $(-1)^{\#P}$. Thus we obtain 
\[
\begin{split}
e(\Hom^<(P, Q\times (0, 1))) 
&= \sum_{X \subset \L(P)} e(\pi_{1*}^{-1}(M_X))  = \sum_{X \subset \L(P)} e(M_X\times \overset{\circ}{F_X}) \\
&= \sum_{X \subset \L(P)} e(M_X) \cdot (-1)^{\#P} = (-1)^{\#P} \cdot e\!\left(\bigsqcup_{X \subset \L(P)}M_X\right) \\
&= (-1)^{\# P}\cdot e(\Hom^{\leq}(P,Q)). 
\end{split}
\]
This completes the proof. 

\subsection{Stanley's reciprocity for order polynomials}

\label{subsec:stanrecip}

In this section, we deduce Stanley's reciprocity \eqref{eq:stanrecip} from Theorem \ref{thm:erecip}. The idea is to take semialgebraic totally ordered posets as the target posets. 

\begin{example}
\label{ex:satoset}
Any semialgebraic set $X\subset\R$ with induced ordering 
is a semialgebraic totally ordered set. 
Furthermore, since $\R^n$ is totally ordered by 
the lexicographic ordering, any semialgebraic set $X\subset\R^n$ admits 
the structure of a semialgebraic totally ordered set. 
\end{example}
The Euler characteristic of $\Hom^{\leq}(P, T)$, with $T$ a semialgebraic totally ordered set, can be computed by using 
the order polynomial $\ord^{\leq(<)}(P, t)$. 
\begin{theorem}[Theorem \ref{thm:intro1} \eqref{iii}]
\label{thm:eorderp}
Let $P$ be a finite poset and $T$ be a semialgebraic totally ordered set. 
Then 
\begin{align}
&e(\Hom^{\leq}(P, T))=\ord^{\leq}(P, e(T)), \label{eq:eorderepoly1} \\
&e(\Hom^{<}(P, T))=\ord^{<}(P, e(T)). \label{eq:eorderepoly2}
\end{align}
\end{theorem}
Before proving Theorem \ref{thm:eorderp}, we need several lemmas 
on the Euler characteristics of configuration spaces. 

\begin{definition}
Let $X$ be a semialgebraic set. The ordered configuration space 
of $n$-points on $X$, denoted by $C_n(X)$, is defined by 
\[
C_n(X)=\{(x_1, \dots, x_n)\in X^n\mid x_i\neq x_j\text{ if }i\neq j\}. 
\]
\end{definition}

\begin{lemma}
$e(C_n(X))=e(X)\cdot(e(X)-1)\cdots(e(X)-n+1)$. 
\end{lemma}
\begin{proof}
It is proved by induction. When $n=1$, it is obvious from $C_1(X)=X$. 
Suppose $n>1$. Consider the projection 
\[
\pi\colon  C_n(X)\longrightarrow C_{n-1}(X), (x_1, \dots, x_n)\longmapsto
(x_1, \dots, x_{n-1}). 
\]
Then the fiber of $\pi$ at the point 
$(x_1, \dots, x_{n-1})\in C_{n-1}(X)$ is 
\[
X\smallsetminus\{x_1, \dots, x_{n-1}\}, 
\]
which has the Euler characteristic 
\[
e(X\smallsetminus\{x_1, \dots, x_{n-1}\})=e(X)-(n-1). 
\]
Therefore, from the inductive assumption, we have 
\[
\begin{split}
e(C_n(X))&=e(C_{n-1}(X))\cdot(e(X)-n+1)\\
&=e(X)\cdot(e(X)-1)\cdots(e(X)-n+1). 
\end{split}
\]
\end{proof}
\begin{remark}
We will give a stronger result later (Theorem \ref{thm:echr} 
and Corollary \ref{cor:kn}). 
\end{remark}

\begin{lemma}
\label{lem:etoset}
Let $T$ be a semialgebraic totally ordered set. Then 
\begin{equation}
e(\Hom^{<}([n], T))=\frac{e(T)\cdot(e(T)-1)\cdots(e(T)-n+1)}{n!}. 
\end{equation}
\end{lemma}

\begin{proof}
The set 
\[
\Hom^<([n], T)=\{(x_1, \dots, x_n)\in T^n\mid 
x_1<\cdots<x_n\} 
\]
is obviously a subset of the configuration space $C_n(T)$. 
Moreover, using the natural action of the symmetric group 
$\mathfrak{S}_n$ on $C_n(T)$ and the fact that $T$ is 
totally ordered, we have 
\[
C_n(T)=\bigsqcup_{\sigma\in\mathfrak{S}_n}
\sigma(\Hom^<([n], T)). 
\]
Since the group action preserves 
the Euler characteristic, we obtain the following. 
\[
e(C_n(T))=n!\cdot e(\Hom^<([n], T)). 
\]
\end{proof}

\begin{proof}[Proof of Theorem \ref{thm:eorderp}]  We fix $\ep \in \{\leq, <\}$. Let $f\in\Hom^{\ep}(P, T)$. Since $P$ is a finite poset, 
the image $f(P)\subset T$ is a finite totally ordered set. 
Suppose $\#f(P)=k$. Then the map $f$ is decomposed as 
\[f\colon 
P
\stackrel{\alpha}{\longrightarrow}
[k]
\stackrel{\beta}{\longrightarrow}
T, 
\]
where $\alpha\colon P\longrightarrow[k]$ is surjective while 
$\beta\colon [k]\longrightarrow T$ is injective. Hence 
$\beta$ can be considered as an element of 
$\Hom^<([k], T)$, and we have the following decomposition, 
\begin{equation}
\label{eq:tosetdec}
\Hom^\ep(P, T)=
\bigsqcup_{k\geq 1}
\Hom^{\ep, \surj}(P, [k])\times
\Hom^<([k], T), 
\end{equation}
where $\Hom^{\ep, \surj}(P, [k])$ is the set of surjective maps 
in $\Hom^{\ep}(P, [k])$. 
By putting $T=[n]$ and then extending $n$ to real numbers $t$,  we obtain the expression for the (strict) order polynomial, 
\begin{equation}
\label{eq:stan}
\ord^{\ep}(P, t)=\sum_{k\geq 1}
\#\Hom^{\ep, \surj}(P, [k])\cdot\frac{t(t-1)\cdots(t-k+1)}{k!},  
\end{equation}
which was already obtained by Stanley \cite[Theorem 1]{sta-chr}. 
Using (\ref{eq:tosetdec}), Lemma \ref{lem:etoset} and 
(\ref{eq:stan}), we have 
\[
\begin{split}
e(\Hom^{\ep}(P, T))
&=
\sum_{k\geq 1}
e(\Hom^{\ep, \surj}(P, [k]))\cdot
e(\Hom^<([k], T))\\ 
&=
\sum_{k\geq 1}
\#\Hom^{\ep, \surj}(P, [k])\cdot
\frac{e(T)(e(T)-1)\cdots(e(T)-k+1)}{k!}
\\ 
&=\ord^{\ep}(P, e(T)). 
\end{split}
\]
This completes the proof of Theorem \ref{thm:eorderp}. 
\end{proof}
\begin{corollary}
\label{cor:stan}
(Stanley's reciprocity \cite{sta-chr}) 
Let $P$ be a finite poset and $n\in\N$. Then 
\begin{equation}
\label{eq:stanley}
\#\Hom^<(P, [n])=(-1)^{\#P}\cdot \ord^{\leq}(P, -n). 
\end{equation}
\end{corollary}

\begin{proof}
Since $\Hom^<(P, [n])$ is a finite poset, the cardinality 
is equal to the Euler characteristic: 
$\#\Hom^<(P, [n])=e(\Hom^<(P, [n]))$. 
We apply the Euler characteristic reciprocity (Theorem \ref{thm:erecip}), 
\[
e(\Hom^<(P, [n]))=(-1)^{\#P}\cdot e(\Hom^{\leq}(P, [n]\times (0,1))). 
\]
Note that $[n]\times (0, 1)$ is a semialgebraic totally ordered set 
(with the lexicographic ordering) with the Euler characteristic 
$e([n]\times (0,1))=-n$. 
Applying Theorem \ref{thm:eorderp}, we have 
\[
e(\Hom^{\leq}(P, [n]\times (0,1)))=
\ord^{\leq}(P, -n), 
\]
which implies (\ref{eq:stanley}). 
\end{proof}

\section{Chromatic and flow polynomials for finite graphs}

\label{sec:chroflow}

In this section, we formulate Euler characteristic reciprocities for 
chromatic polynomials of finite simple graphs and for flow polynomials of finite oriented graphs. 

\subsection{Chromatic polynomials}

\label{subsec:ch}

Let $G=(V, E)$ be a finite simple graph with vertex set $V$ and 
(un-oriented) edge set $E$. The chromatic polynomial is 
a polynomial $\chi(G, t)\in\Z[t]$ which satisfies 
\[
\chi(G, n)=\#\{c\colon V\longrightarrow[n]\mid v_1v_2\in E\Longrightarrow 
c(v_1)\neq c(v_2)\}, 
\]
for all $n>0$. The chromatic polynomial is also characterized 
by the following properties: 
\begin{itemize}
\item if $E=\emptyset$ then $\chi(G, t)=t^{\#V}$;  
\item if $e\in E$, then $\chi(G, t)=\chi(G-e, t)-\chi(G/e, t)$, 
where $G-e$ and $G/e$ are the deletion and the contraction with 
respect to the edge $e$, respectively. 
\end{itemize}
\begin{definition}
\label{def:kasen}
For a set $X$, define the set of vertex coloring with $X$ (or 
the graph configuration space) by 
\begin{equation}
\kasen(G, X)=\{c\colon V\longrightarrow X\mid v_1v_2\in E\Longrightarrow 
c(v_1)\neq c(v_2)\}. 
\end{equation}
\end{definition}
The assignment $X\longmapsto \kasen(G, X)$ can be considered as 
a functor (\cite{yos-chr}). The space $\kasen(G, X)$ is also 
called the graph (generalized) configuration space (\cite{eas-hug}). 

The chromatic polynomial $\chi(G,t)\in\Z[t]$ satisfies 
$\chi(G,n)=\#\kasen(G,[n])$ for all $n\in\N$. 

In this section, we investigate the Euler characteristic aspects 
of the chromatic polynomial for a finite simple graph. 

When $X$ is a semialgebraic set, $\kasen(G, X)$ is also a semialgebraic set. 
The following result generalizes \cite[Theorem 2]{eas-hug}, where the result is shown when $X$ is a complex projective space.  

\begin{theorem}
\label{thm:echr}
Let $G=(V, E)$ be a finite simple graph and $X$ be a semialgebraic set. Then 
\begin{equation}
\label{eq:echr}
e(\kasen(G, X))=\chi(G, e(X)). 
\end{equation}
\end{theorem}

\begin{proof}
This result is proved by induction on $\#E$. When $E=\emptyset$, 
$e(\kasen(G, X))=e(X^{\#V})=e(X)^{\#V}=\chi(G, e(X))$. Suppose 
$e\in E$. Then we can prove 
\begin{equation}
\label{eq:additivity}
\kasen(G-e, X)\simeq\kasen(G, X)\sqcup\kasen(G/e, X). 
\end{equation}
Using the additivity of the Euler characteristic and 
the recursive relation for the chromatic polynomial, 
we obtain (\ref{eq:echr}). 
\end{proof}
Note that for the complete graph $G=K_n$, 
$\kasen(K_n, X)$ is identical to the configuration space 
$C_n(X)$ of $n$-points. 
Applying Theorem \ref{thm:echr} to the complete graph 
$K_n$ (which has the chromatic polynomial $\chi(K_n, t)=
t(t-1)\cdots(t-n+1)$), we have the following. 
\begin{corollary}
\label{cor:kn}
$e(C_n(X))=e(X)(e(X)-1)\cdots(e(X)-n+1)$. 
\end{corollary}

To formulate the reciprocity for chromatic polynomials, 
we recall the notion of acyclic orientations on a graph $G$. 

Let $G=(V, E)$ be a finite simple graph. The set of edges $E$ 
can be considered as a subset of 
\[
(V\times V\smallsetminus\Delta)/{\mathfrak{S}}_2, 
\]
where $\Delta=\{(v, v)\mid v\in V\}$ is the diagonal subset and 
$\mathfrak{S}_2$ acts on $V\times V$ by transposition. 
There is a natural projection 
\[
\pi\colon 
V\times V\smallsetminus\Delta\longrightarrow 
(V\times V\smallsetminus\Delta)/{\mathfrak{S}}_2. 
\]
An edge orientation on $G$ is a subset $\tilE\subset
V\times V\smallsetminus\Delta$ such that $\pi|_{\tilE}\colon 
\tilE\stackrel{\simeq}{\longrightarrow}E$ is a 
bijection. An orientation $\tilE$ is said to 
contain an oriented cycle, if there exists a 
cyclic sequence $(v_1, v_2), (v_2, v_3), \dots, (v_{n-1}, v_n), 
(v_n, v_1)\in\tilE$ for some $n>2$. 
The orientation $\tilE$ is called \emph{acyclic} if 
it does not contain oriented cycles. 

\begin{definition}
\label{def:aoc}
Let $G=(V, E)$ be a finite simple graph. Fix an acyclic orientation 
$\tilE\subset V\times V\smallsetminus\Delta$. 
Let $T$ be a totally ordered set. 
\begin{itemize}
\item[(a)] 
A map $c\colon V\longrightarrow T$ is said to be \emph{compatible} with $\tilE$ if
\[
(v, v')\in\tilE\Longrightarrow c(v)\leq c(v'). 
\]
\item[(b)] 
A map $c\colon V\longrightarrow T$ is said to be \emph{strictly compatible} with $\tilE$ if
\[
(v, v')\in\tilE\Longrightarrow c(v)< c(v'). 
\]
\end{itemize}
\end{definition}
We denote the sets of all pairs of an acyclic orientation with 
a compatible map, and with a strictly compatible map, by 
\[
\AOC^{\leq}(G, T):=
\left\{(\tilE, c)
\left|
\begin{array}{ll}
\tilE\text{ is an acyclic orientation, and }c\colon V\rightarrow T\\
\text{ is a map compatible with $\tilE$}
\end{array}
\right.
\right\}, 
\]
and
\[
\AOC^<(G, T):=
\left\{(\tilE, c)
\left|
\begin{array}{ll}
\tilE\text{ is an acyclic orientation, and }c\colon V\rightarrow T\\
\text{ is a map strictly compatible with $\tilE$}
\end{array}
\right.
\right\}, 
\]
respectively. 

If $T$ is a semialgebraic totally ordered set, then 
these spaces possess the structures of semialgebraic sets. 
We will see a reciprocity between 
these two spaces from which Stanley's reciprocity for chromatic polynomials 
is deduced. 

It is straight forward that 
$\AOC^{<}(G, T)$ can be identified with $\kasen(G, T)$. 
In particular, we have 
\begin{equation}
\label{eq:aocchi}
e(\AOC^{<}(G, T))=\chi(G, e(T)). 
\end{equation}

We formulate a reciprocity for chromatic polynomials in terms of Euler characteristics. 

\begin{theorem}
\label{thm:aocrecip}
Let $G=(V, E)$ be a finite simple graph and $T$ be a semialgebraic 
totally ordered set. Then 
\begin{align}
e(\AOC^{\leq}(G, T))=(-1)^{\#V}\cdot e(\AOC^{<}(G, T\times(0,1))), \label{eq:aocrecip1} \\
e(\AOC^{<}(G, T))=(-1)^{\#V}\cdot e(\AOC^{\leq}(G, T\times(0,1))). \label{eq:aocrecip2}
\end{align}
\end{theorem}
To prove Theorem \ref{thm:aocrecip}, we give alternative descriptions of 
$\AOC^{\leq(<)}(G, T)$ in terms of poset homomorphisms and 
graph configuration spaces. Let $\tilE$ be an acyclic orientation 
of $G=(V, E)$. Then $\tilE$ determines an ordering on $V$, called the transitive closure of $\tilE$, defined by 
\[
v<v'\Longleftrightarrow
\exists v_0, \dots, v_n\in V \text{ s.t. }
\left\{
\begin{array}{l}
v=v_0, v'=v_n, \text{ and}\\
(v_{i-1}, v_i)\in \tilE\text{ for }1\leq i\leq n. 
\end{array}
\right.
\]
This ordering defines a poset which we denote by $P(V, \tilE)$. 

A map $c\colon V\longrightarrow T$ is compatible with $\tilE$ if and only if 
$c$ is an increasing map from $P(V, \tilE)$ to $T$. Hence the set 
of maps compatible with $\tilE$ is identified with 
$\Hom^{\leq}(P(V, \tilE), T)$. We have the following decomposition. 
\begin{equation}
\label{eq:decaoc1}
\AOC^{\leq}(G, T)\simeq
\bigsqcup_{\tilE: \text{ acyclic ori.}}
\Hom^{\leq}(P(V, \tilE), T). 
\end{equation}
Similarly, $\AOC^<(G, T)$ is decomposed as follows. 
\begin{equation}
\label{eq:decaoc2}
\AOC^{<}(G, T)\simeq
\bigsqcup_{\tilE: \text{ acyclic ori.}}
\Hom^{<}(P(V, \tilE), T). 
\end{equation}

\begin{proof}[Proof of Theorem \ref{thm:aocrecip}]
We prove \eqref{eq:aocrecip1}. Using the above decompositions \eqref{eq:decaoc1} and \eqref{eq:decaoc2} together with Theorem \ref{thm:erecip}, 
we obtain 
\[
\begin{split}
e(\AOC^{\leq}(G, T))
&=
e\left(
\bigsqcup_{\tilE: \text{ acyclic ori.}}
\Hom^{\leq}(P(V, \tilE), T)\right)\\
&=
\sum_{\tilE: \text{ acyclic ori.}}
e\left(
\Hom^{\leq}(P(V, \tilE), T)\right)
\\
&=
(-1)^{\#V}\cdot
\sum_{\tilE: \text{ acyclic ori.}}
e\left(
\Hom^{<}(P(V, \tilE), T\times(0,1))\right)
\\
&=
(-1)^{\#V}\cdot
e\left(
\bigsqcup_{\tilE: \text{ acyclic ori.}}
\Hom^{<}(P(V, \tilE), T\times(0,1))\right)
\\
&=
(-1)^{\#V}\cdot
e(\AOC^{<}(G, T\times(0,1))). 
\end{split}
\]
This completes the proof. The second formula \eqref{eq:aocrecip2} is proved similarly. 
\end{proof} 

We deduce Stanley's reciprocity on chromatic polynomials (\cite{sta-acy}). 
Applying Theorem \ref{thm:aocrecip} and (\ref{eq:aocchi}) shows that (note that 
$T\times (0,1)$ is also a semialgebraic totally ordered set)
\[
\begin{split}
e(\AOC^{\leq}(G, T))
&=
(-1)^{\#V}\cdot
e(\AOC^{<}(G, T\times(0,1)))\\
&=
(-1)^{\#V}\cdot
\chi(G, e(T\times(0,1)
\\
&=
(-1)^{\#V}\cdot
\chi(G, -e(T)). 
\end{split}
\]
Putting $T=[n]$, we have the following Stanley's reciprocity. 
\begin{corollary}
\label{cor:chrec}
Let $G=(V, E)$ be a finite simple graph and $n\in\N$. Then 
\[
\#\AOC^{\leq}(G, [n])=(-1)^{\#V}\cdot
\chi(G, -n). 
\]
\end{corollary}

\subsection{Flow polynomials}

\label{subsec:fl}
This section treats finite oriented graphs that are allowed to have distinguished multiple edges and loops. Our object is a tuple $G=(V, \cE, h, t)$ where  
$V$ and $\cE$ are finite sets and $h\colon \cE\longrightarrow V$ and 
$t\colon\cE\longrightarrow V$ are maps. An element of $V$ is called a vertex and an element of $\cE$ is called an edge. For an edge $e\in \cE$, $h(e)\in V$ is called the head and $t(e)\in V$ is called the tail. 
An edge $e\in \cE$ is a loop if $h(e)=t(e)$.  In Figure \ref{fig:OG}, the oriented graph $G$ has five edges $e_1,\dots,e_5$ and their orientations are described by $h(e_1)=h(e_2)=t(e_3)=x$, $t(e_1)=t(e_2)=h(e_3)=h(e_4)=y$ and $t(e_4)=h(e_5)=t(e_5)=z$. 

\begin{figure}[b]
\centering
\begin{tikzpicture}[scale=2.5]
\draw (0,0) node[v](1){$x$}; 
\draw (1,0) node[v](2){$y$}; 
\draw (2,0) node[v](3){$z$};
\draw[thick,->] (1)--(2) node at (0.5,0.08) {$e_2$}; 
\draw[thick,->] (1) to [out=45, in=135] (2) node at (0.5,0.35) {$e_1$};
\draw[thick,->] (2) to [out=225, in=-45] (1) node at (0.5,-0.35) {$e_3$};
\draw[thick,->] (2)--(3) node at (1.5,0.08) {$e_4$};
\draw[thick,->] (2.05,-0.1) arc [start angle = -145, end angle = 145, radius = 0.19] node at (2.52,0) {$e_5$};
\end{tikzpicture}
\caption{An oriented graph.}
\label{fig:OG}
\end{figure}
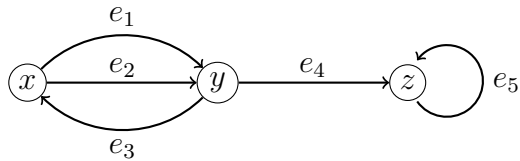

An oriented graph $G$ can also be seen as a $1$-dimensional 
CW-complex. The number of connected components and the $1$-st Betti numbers 
are denoted by $b_0(G)$ and $b_1(G)$, respectively. Note that 
$b_0(G)-b_1(G)=\#V-\#\cE$. An edge $e\in \cE$ is called a coloop if 
$b_0(G\smallsetminus e)=b_0(G)+1$. The graph in Figure \ref{fig:OG} has the unique coloop $e_4$.

Let $\mathcal{A}$ be an abelian group. The map $f\colon \cE\longrightarrow\mathcal{A}$ 
is called an $\mathcal{A}$-flow if $f$ satisfies 
\begin{equation}
\sum_{e: h(e)=v}f(e)
=
\sum_{e: t(e)=v}f(e)
\end{equation}
for all $v\in V$. Let $f$ be an $\mathcal{A}$-flow. Denote 
$\Supp(f)=\{e\in \cE\mid f(e)\neq 0\}$. An $\mathcal{A}$-flow 
is called nowhere zero if $\Supp(f)=\cE$. 
The set of all 
$\mathcal{A}$-flows 
and nowhere zero $\mathcal{A}$-flows are denoted by 
$\mathcal{F}(G, \mathcal{A})$ 
and 
$\mathcal{F}^0(G, \mathcal{A})$, respectively. 

Let $\mathcal{A}$ be a semialgebraic abelian group. Then clearly 
$\mathcal{F}^0(G, \mathcal{A})$ possesses a structure of a semialgebraic 
set. 

The flow polynomial is a polynomial $\phi_G(t)\in\Z[t]$ which satisfies 
\[
\phi_G(k)=\#\mathcal{F}^0(G, \Z/k\Z), 
\]
for all $k>0$. The flow polynomial is also characterized by the 
following properties: 
\begin{itemize}
\item 
if $\cE=\emptyset$, then $\phi_G(t)=1$; 
\item 
if $e\in \cE$ is a loop, then $\phi_G(t)=(t-1)\phi_{G\smallsetminus e}(t)$; 
\item 
if $e\in \cE$ is a coloop, then $\phi_G(t)=0$; 
\item 
if $e\in \cE$ is neither a loop nor a coloop, then 
$\phi_G(t)=\phi_{G/e}(t)-\phi_{G\smallsetminus e}(t)$. 
\end{itemize}

\begin{proposition}
\label{prop:additivity}
Let $G$ be a finite oriented graph, and $\mathcal{A}$ be a 
semialgebraic abelian group. 
\begin{itemize}
\item[(a)] 
If $e\in \cE$ is a loop, 
then $\mathcal{F}^0(G, \mathcal{A})\simeq 
(\mathcal{A}\smallsetminus\{0\})\times
\mathcal{F}^0(G\smallsetminus e, \mathcal{A})$. 
\item[(b)] 
If $e\in \cE$ is a coloop, then $\mathcal{F}^0(G, \mathcal{A})=\emptyset$.  
\item[({c})]   
If $e\in \cE$ is neither a loop nor a coloop, then 
$\mathcal{F}^0(G/e, \mathcal{A})\simeq 
\mathcal{F}^0(G, \mathcal{A})\sqcup
\mathcal{F}^0(G\smallsetminus e, \mathcal{A})$. 
\end{itemize}
\end{proposition}
\begin{proof}
Straightforward. 
\end{proof}

\begin{theorem}
Let $G$ be a finite oriented graph and 
$\mathcal{A}$ be a semialgebraic abelian group. 
Then $e(\mathcal{F}^0(G, \mathcal{A}))=\phi_G(e(\mathcal{A}))$. 
\end{theorem}
\begin{proof}
Using Proposition \ref{prop:additivity}, it is proved by 
induction on the number of edges. 
(See Theorem \ref{thm:echr}.) 
\end{proof}

An oriented graph $G$ is called totally cyclic if every edge is 
contained in an oriented cycle. Let $\sigma\subset \cE$ be a subset of 
edges and denote by ${}_\sigma G$ the reorientation of $G$ along $\sigma$. 
A subset $\sigma\subset \cE$ is a totally cyclic reorientation if 
${}_\sigma G$ is totally cyclic. 

Let us denote by 
$\FTC(G, \mathcal{A})$ 
the set of all pairs $(f, \sigma)$ 
of the flow $f$ and totally cyclic reorientation 
$\sigma\subset \cE\smallsetminus\Supp(f)$. Namely, 
\[
\FTC(G, \mathcal{A})=
\left\{
(f, \sigma)
\left|
\begin{array}{l}
f\in\mathcal{F}(G, \mathcal{A}), \mbox{ and 
$\sigma\subset \cE\smallsetminus\Supp(f)$ is a}\\
\mbox{totally cyclic reorientation for }G_{/\Supp(f)} 
\end{array}
\right.
\right\}. 
\]
For each subset $\sigma \subset \cE$, the set of all $f$ with 
$(f, \sigma)\in\FTC(G, \mathcal{A})$ forms a semialgebraic 
subset of $\mathcal{F}(G, \mathcal{A})$. Therefore 
$\FTC(G, \mathcal{A})$ possesses a structure of semialgebraic 
set. Let us define $-\mathcal{A}$ by 
\[
-\mathcal{A}:=\mathcal{A}\times\R. 
\]
The following is proved along the same lines of the proof presented in 
\cite[Appendix A]{bre-san}, which can be considered as a 
Breuer-Sanyal's reciprocity at the level of Euler characteristic. 

\begin{theorem}
\label{thm:flowrecip}
Let $G$ be a finite oriented graph and $\mathcal{A}$ be a semialgebraic 
abelian group. Then 
\[
e(\FTC(G, \pm\mathcal{A}))=
(-1)^{b_1(G)}e(\mathcal{F}^0(G, \mp\mathcal{A})). 
\]
\end{theorem}

\medskip

\noindent
{\bf Acknowledgements.} 
T. H. was supported by Grant-in-Aid for Young Scientists (B) 15K17549, JSPS and M. Y. was partially supported by Grant-in-Aid for Scientific Research (C) 25400060, JSPS. The authors thank Matthias Beck for informing them of the flow polynomial reciprocity.

\end{document}